\documentclass[12pt]{amsart}

\usepackage{etex}
\usepackage[usenames,dvipsnames]{pstricks} 
\usepackage{epsfig}
\usepackage{graphicx,color}
\usepackage{geometry}
\usepackage[all]{xy}
\usepackage{amssymb,amscd}
\usepackage{cite}
\usepackage{fullpage}
\usepackage{marvosym}
\xyoption{poly}
\usepackage{url}
\usepackage{comment}
\usepackage{float}
\usepackage{bm}

\usepackage{tikz}
\usepackage{tikz-cd}
\usetikzlibrary{decorations.pathmorphing}
\newtheorem{introtheorem}{Theorem}

\newtheorem{theorem}{Theorem}[section]
\newtheorem{lemma}[theorem]{Lemma}
\newtheorem{proposition}[theorem]{Proposition}

\theoremstyle{definition}
\newtheorem{definition}[theorem]{Definition}

\newtheorem{remark}[theorem]{Remark}

\newtheorem*{question*}{Question}

\newtheorem*{questions*}{Questions}

\newtheorem*{steps*}{Answer/steps}

\newtheorem*{progress*}{Progress}

\newtheorem*{classification*}{Classification}

\newtheorem*{construction*}{Classification}
\newtheorem*{example*}{Example}

\newtheorem*{remark*}{Remark}
\newtheorem*{remarks*}{Remarks}
\newtheorem*{definition*}{Definition}

\usepackage{calrsfs}
\usepackage{url}
\usepackage{longtable}
\usepackage[OT2, T1]{fontenc}
\usepackage{textcomp}
\usepackage{times}
\usepackage[scaled=0.92]{helvet}

\newcommand{\Z}{\mathbb{Z}}

\newcommand{\F}{\mathbb{F}}

\newcommand{\X}{\mathcal{X}}

\DeclareMathOperator{\GL}{GL}

\DeclareMathOperator{\Aut}{Aut}

\DeclareSymbolFont{cyrletters}{OT2}{wncyr}{m}{n}
\DeclareMathSymbol{\Sha}{\mathalpha}{cyrletters}{"58}

\makeatletter

\def\greekbolds#1{%
 \@for\next:=#1\do{%
    \def\X##1;{%
     \expandafter\def\csname V##1\endcsname{\boldsymbol{\csname##1\endcsname}}
     }
   \expandafter\X\next;
  }
}

\greekbolds{alpha,beta,iota,gamma,lambda,nu,eta,Gamma,varsigma,Lambda}

\def\make@bb#1{\expandafter\def
  \csname bb#1\endcsname{{\mathbb{#1}}}\ignorespaces}

\def\make@bbm#1{\expandafter\def
  \csname bb#1\endcsname{{\mathbbm{#1}}}\ignorespaces}

\def\make@bf#1{\expandafter\def\csname bf#1\endcsname{{\bf
      #1}}\ignorespaces} 

\def\make@gr#1{\expandafter\def
  \csname gr#1\endcsname{{\mathfrak{#1}}}\ignorespaces}

\def\make@scr#1{\expandafter\def
  \csname scr#1\endcsname{{\mathscr{#1}}}\ignorespaces}

\def\make@cal#1{\expandafter\def\csname cal#1\endcsname{{\mathcal
      #1}}\ignorespaces} 
\def\do@Letters#1{#1A #1B #1C #1D #1E #1F #1G #1H #1I #1J #1K #1L #1M
                 #1N #1O #1P #1Q #1R #1S #1T #1U #1V #1W #1X #1Y #1Z}
\def\do@letters#1{#1a #1b #1c #1d #1e #1f #1g #1h #1i #1j #1k #1l #1m
                 #1n #1o #1p #1q #1r #1s #1t #1u #1v #1w #1x #1y #1z}
\do@Letters\make@bb   \do@letters\make@bbm
\do@Letters\make@cal  
\do@Letters\make@scr 
\do@Letters\make@bf \do@letters\make@bf   
\do@Letters\make@gr   \do@letters\make@gr
\makeatother

\def\ol{\overline}
\def\wt{\widetilde}

\def\onto{\twoheadrightarrow}

\newcommand{\<}{\langle}   %\< is not defined yet.
\renewcommand{\>}{\rangle} %\> is already defined.

\newcommand{\isoto}{\stackrel{\sim}{\longrightarrow}}

\def\Spec{{\rm Spec}\,}

\def\Fpbar{\overline{\bbF}_p}

\def\Zp{{\bbZ}_p}

\newcommand{\A}{\mathbb A}    % for adele

\def\makeop#1{\expandafter\def\csname#1\endcsname
  {\mathop{\rm #1}\nolimits}\ignorespaces}
\makeop{Hom}   \makeop{End}   \makeop{Aut}   \makeop{Isom}  \makeop{Pic} 
\makeop{Gal}   \makeop{ord}   \makeop{Char}  \makeop{Div}   \makeop{Lie} 
\makeop{PGL}   \makeop{Corr}  \makeop{PSL}   \makeop{sgn}   \makeop{Spf}
\makeop{Spec}  \makeop{Tr}    \makeop{Nr}    \makeop{Fr}    \makeop{disc}
\makeop{Proj}  \makeop{supp}  \makeop{ker}   \makeop{im}    \makeop{dom}
\makeop{coker} \makeop{Stab}  \makeop{SO}    \makeop{SL}    \makeop{SL}
\makeop{Cl}    \makeop{cond}  \makeop{Br}    \makeop{inv}   \makeop{rank}
\makeop{id}    \makeop{Fil}   \makeop{Frac}  \makeop{GL}    \makeop{SU}
\makeop{Nrd}   \makeop{Sp}    \makeop{Tr}    \makeop{Trd}   \makeop{diag}
\makeop{Res}   \makeop{ind}   \makeop{depth} \makeop{Tr}    \makeop{st}
\makeop{Ad}    \makeop{Int}   \makeop{tr}    \makeop{Sym}   \makeop{can}
\makeop{length}\makeop{SO}    \makeop{torsion} \makeop{GSp} \makeop{Ker}
\makeop{Adm}   \makeop{Mat}

\newcommand{\dieu}{Dieudonn\'{e} }

\DeclareMathSymbol{\twoheadrightarrow} {\mathrel}{AMSa}{"10}

\DeclareMathOperator{\pr}{pr}

\def\sfF{\mathsf{F}}
\def\sfV{\mathsf{V}}
\begin{document}

\title{Oort's Conjecture on Supersingular abelian varieties in odd characteristic}

\author{Valentijn Karemaker}
\address{Korteweg-de Vries Institute for Mathematics, University of Amsterdam, The Netherlands}
\email{V.Z.Karemaker@uva.nl}

\author{Chia-Fu Yu}
\address{Institute of Mathematics, Academia  Sinica and National Center for Theoretic Sciences, Taipei, Taiwan}
\email{chiafu@math.sinica.edu.tw}

%\date{\today}

 \begin{abstract} 
 Oort's conjecture asserts that for any $g\ge 2$ and any prime $p$, every geometric generic member in the supersingular locus $\calS_g$ has automorphism group $\{\pm1 \}$. This has been proved very recently by Viehmann in full generality with previously known counterexamples for $p=2$ and $g=2,3$. We construct, for any $g\ge 3$, a closed subvariety of dimension $g-1$ which contains an open dense subset $\mathcal{U}$ of $a$-invariant $g-2$ such that $\mathcal{U}$ meets every irreducible component of $\calS_g$ and every geometric point in $\mathcal{U}$ has automorphism group $\{\pm 1\}$ if $p>2$. This gives an independent proof of Oort's conjecture for $p>2$. When $g>3$, so $a = g-2>1$, this provides complementary information to the $a=1$ locus investigated in Viehmann's work on how automorphism groups interact with the geometry.
 \end{abstract}

\maketitle
\setcounter{tocdepth}{2}

\section{Introduction}
Let $g\ge 1$ be a positive integer and $p$ a prime number. 
Let $\calA_g$ be the moduli space over $\Fpbar$ of $g$-dimensional principally polarised abelian varieties, and let $\calS_g$ be the supersingular locus of $\calA_g$. 
Let $k$ be an algebraically closed field of characteristic $p$. For each element $(X,\lambda)$ in $\calA_g(k)$, it is a fundamental question to understand what the endomorphism ring $\End(X)$ of $X$ and the automorphism group $\Aut(X,\lambda)$ of $(X,\lambda)$ may be. Moreover, it is also interesting to understand how these arithmetic invariants vary in the moduli space $\calA_g$ or in a subvariety, for example, in $\calS_g$. 

Chai and Oort~\cite{COirr} showed that for any prime $\ell\neq p$, the $\ell$-adic monodromy attached to any non-supersingular central leaf $\calC \subseteq \calA_g$ is surjective and that $\calC$ is irreducible. It follows that the geometric generic member $(X_{\bar \eta},\lambda_{\bar \eta})$ of $\calC$ has endomorphism ring $\Z$ and hence automorphism group $\{\pm 1\}$. Using this, every geometric generic member $(X_{\bar \eta},\lambda_{\bar \eta})$ of  either a non-supersingular Newton stratum or a non-supersingular Ekedahl-Oort (EO) stratum (i.e., an Ekedahl-Oort stratum that is not entirely contained in $\calS_g$ ) also shares the same property. 

One may ask what one can say for the supersingular case; it has been well understood that supersingular strata behave quite differently from non-supersingular strata. Indeed, the $\ell$-adic monodromy for every irreducible component of $\mathcal{S}_g$ has only finite image (though by contrast, the $\ell$-adic monodromy for the whole supersingular locus has very large image),  and every geometric generic member in $\calS_g$ has an endomorphism ring of $\Z$-rank $4g^2$. Despite these different arithmetic features, Oort's conjecture \cite[Question 4]{edixhoven-moonen-oort} asserts that when $g\ge 2$, every geometric generic member in $\calS_g$ still has automorphism group $\{\pm1 \}$.

For $g=2$ and $p>2$, Oort's conjecture has been proved independently by Ibukiyama~\cite{ibukiyama}, and by the first author and Pries \cite{karemaker-pries}, with a counterexample in  $p=2$. 
In~\cite{karemaker-yobuko-yu} the present authors and Yobuko studied mass strata on the supersingular locus $\calS_3$ and obtained explicit mass formulae on each stratum. 
The authors also show that on the maximal mass stratum, each geometric point has automorphism group $\{\pm 1\}$ if $p>2$, and $\{\pm 1\}^3$ if $p=2$, confirming Oort's conjecture for $g=3$, with again a counterexample in $p=2$. 

In~\cite{karemaker-yu:SSEOOC}, the present authors explore arithmetic invariants on the union $\calS_g^{\rm eo}$ of all supersingular EO strata. They introduce a stratification on $\calS_g^{\rm eo}$ which refines the mass stratification (i.e., the mass function is constant on each new stratum) and obtain mass formulae on each stratum. The authors further show that in the maximal stratum (which is an open dense subset of the maximal supersingular EO stratum), every geometric point has automorphism group $\{\pm 1\}$ when $g$ is even and $p>3$; see~\cite[Theorem~6.17]{karemaker-yu:SSEOOC}. This confirms the validity of Oort's conjecture under these conditions, since every irreducible component of $\calS_g$ meets this maximal stratum; cf.~\cite[Proposition~6.22]{karemaker-yu:SSEOOC}. The authors also prove Oort's conjecture for $\calS_4$, exploiting a description of the locus with $a$-number one due to Harashita (cf.~\cite[Theorem~7.1]{karemaker-yu:SSEOOC}). In addition, they determine a bound for the torsion of the automorphism groups of quasi-polarised supersingular $p$-divisible groups of any dimension  (cf.~\cite[Lemma~6.16]{karemaker-yu:SSEOOC}), which provides a strategy for computing automorphism groups modulo ``principal congruence subgroups'' (denoted $V_{p,s}$ in \emph{loc.~cit.~}) of level $s=1, 2, 3$.

Oort's conjecture is proved in full generality by Viehmann~\cite{viehmann:oort}, except when $g=2$ or $3$ and $p=2$. 
Using a similar local approach, the author reformulates the problem into one for generic quasi-polarised $p$-divisible groups in the Rapoport-Zink space. To determine the torsion of the automorphism groups, the author first provides an explicit description of the locus with $a$-number one in the Rapoport-Zink space, and successively computes the automorphism groups modulo the above-mentioned principal congruence subgroups of level s=1, 2, 3, respectively, through carefully analysing the defining equations. 
The description for $g=4$ is similar to Harashita's and is adapted from~\cite[Section 7]{karemaker-yu:SSEOOC}. Viehmann also bounds the torsion of automorphism groups of generic non-polarised supersingular $p$-divisible groups and deduces an analogue of Oort's conjecture for unitary Shimura varieties with a splitting condition at $p$. \\

In the present paper we proceed with the ``closed subvariety'' approach as in~\cite{karemaker-yu:SSEOOC}. Namely, we construct a closed subvariety $\calT_g$ of $\mathcal{S}_g$ for $g\ge 3$ and an open dense subset $\calU$, which are $\ell$-adic Hecke invariant for any prime $\ell\neq p$, so that every polarised supersingular abelian variety landing in $\calU$ has automorphism group $\{\pm 1\}$ provided that $p>2$. As a consequence, we give a different proof of  Oort's conjecture for $p>2$ and any $g \geq 3$. The subvariety $\calT_g$ comes from modifying the moduli space $\calP_{g,\eta}$ of polarised flag type quotients (PTFQs) constructed by Li and Oort; it has dimension $g-1$ and generic $a$-number $g-2$. Furthermore, its irreducible components are parametrised by 
principal polarisations of a superspecial abelian variety. 

Let $g\ge 3$ and $\mu$ be a principal polarisation on the unique superspecial abelian variety $E^g$ over $k$ with a canonical model $E_0^g$ defined over $\F_{p^2}$. Let $\calP_\mu$ be the moduli space over $\F_{p^2}$ parametrising two-step chains of polarised supersingular abelian varieties 
\[ (Y_\bullet,\lambda_\bullet): (Y_2, \lambda_2)=(E^g,p\mu) \xrightarrow{\rho_{2}}  (Y_1,\lambda_1) \xrightarrow{\rho_{1}} (Y_0,\lambda_0), \]
where $\rho_2$ and $\rho_1$ are polarised isogenies with kernels $\ker\rho_i$ which are $\alpha$-groups  of rank $p^{g-1}$ and $p$, respectively, that satisfy similar conditions to the moduli space $\calP_{g,\eta}$ in~\cite{lioort}; see Definition~\ref{def:Pmu}. Similarly to~\cite[Section~9.4]{lioort}, the map sending an object in $\calP_\mu$ to the isogeny $\rho_2: (E^g,p\mu) \to   (Y_1,\lambda_1)$ endows $\calP_\mu$ with the structure of a $\bbP^1$-bundle over the Fermat hypersurface
\[ \pi:\calP_\mu\longrightarrow \calF:=V(X_1^{p
+1} +\dots +X_g^{p+1})\subseteq \bbP^{g-1}_{\F_{p^2}},
\]
that admits a section with image $T$. 
Let $\calF^0\subseteq \calF$ be the complement of the union of  all $\F_{p^2}$-rational conics in $\bbP^{g-1}$.
Put $\calP_\mu':=\calP_\mu \setminus T$ and $\calP'_{\calF^0}:=\calP'_\mu \times_\calF \calF^0\subseteq \calP'_\mu$, the open subset restricted to $\calF^0$. Finally, let $\calD\subseteq \calP'_{\mu, \calF^0}$ be the horizontal divisor defined in Definition~\ref{def:D}. Then our main result is as follows.

\begin{introtheorem}[Theorem~\ref{thm:OCPmu}]\label{thm:A}
    Let $g \ge 3$ and $p>2$. Then for any chain $(Y_\bullet, \lambda_\bullet)$ of polarised supersingular abelian varieties corresponding to a $k$-point $y$ in $\calP'_{\mu, \calF^0}\setminus \calD$, 
    one has 
    $\Aut(Y_0,\lambda_0)=\{\pm 1\}$.    
\end{introtheorem}

Let $\pr_0: \calP_\mu \to \calS_g$ be the natural projection $(Y_\bullet, \lambda_\bullet)\mapsto (Y_0,\lambda_0)$ and define 
\[ \calT_g:=\bigcup_{\mu \in {\rm Pol}(E_0^g)} \pr_0(\calP_\mu), \]
where $\mu$ runs through all principal polarisations on $E_0^g$. Then $\calT_g$ is a closed subvariety of $\calS_g$. 

\begin{introtheorem}\label{thm:B}
    The closed subvariety $\calT_g\subseteq \calS_g$, for $g\ge 3$, has equi-dimension $g-1$ with generic points of $a$-number $g-2$ and it is $\ell$-adic Hecke invariant for any prime $\ell\neq p$. Moreover, Oort's conjecture for $\calT_g$ holds true provided that $p>2$. 
\end{introtheorem}

Using the transitivity of the prime-to-$p$ Hecke correspondences on the irreducible components of $\calS_g$ \cite[Proposition~6.22]{karemaker-yu:SSEOOC}, we deduce Oort's conjecture for $p>2$ from Theorem~\ref{thm:B}:

\begin{introtheorem}
    \label{thm:OCpodd}
    Oort's conjecture holds true for any $g\ge 2$ and any prime $p>2$.
\end{introtheorem}

\noindent The outline of the paper is as follows. In Section 2, we construct the moduli space $\calP_\mu$ which parametrises certain chains of polarised supersingular abelian varieties of length 2, and explore its geometric properties and the relation with $a$-invariants. In Section 3, we prove Theorems~A and~B. \\

\noindent {\bfseries Acknowledgements.}
    The authors are grateful to Akio Tamagawa for helpful discussions, in particular on the proof of Theorem~\ref{thm:OCPmu}. They also thank Eva Viehmann for explaining her results in~\cite{viehmann:oort}.
    The first author was partially supported by the Dutch Research Council (NWO) through grant VI.Vidi.223.028. 
    The second author was partially supported by the National Science and Technology Council (NSTC) grant  114-2115-M-001-001 and the Academia Sinica IVA grant AS-IA-112-M01.

\section{The moduli space $\calP_{\mu}$}

\subsection{Set-up and arithmetic of $\mathcal{P}_{\mu}$}\

Let $k$ denote an algebraically closed field of characteristic $p$. Let $W=W(k)$ the ring of Witt vectors over $k$ and let $\sigma$ denote the Frobenius map on $W$ and its fraction field $W[\frac{1}{p}]$ induced by $x\mapsto x^p$ on $k$.
We fix a supersingular elliptic curve $E_0$ over $\F_{p^2}$ whose Frobenius endomorphism is $-p$. Let ${\rm Pol}(E_0^g)$ denote the set of isomorphism classes of principal polarisations of $E_0^g$.

\begin{definition}\label{def:Pmu}
For each $g\ge 3$ and $\mu\in {\rm Pol}(E_0^g)$, let 
\[ \calP_{\mu}:{(\F_{p^2}{\rm-Sch})} \to {\rm (Set)} \]
be the functor from the category of $\F_{p^2}$-schemes to the category of sets, which sends an $\F_{p^2}$-scheme $S$ to the set of isomorphism classes of chains of polarised abelian schemes of relative dimension $g$ over $S$  
\[ (Y_\bullet,\lambda_\bullet): (Y_2, \lambda_2) \xrightarrow{\rho_{2}}  (Y_1,\lambda_1) \xrightarrow{\rho_{1}} (Y_0,\lambda_0)   \]
such that 
\begin{itemize}
    \item [(i)] $(Y_2,\lambda_2)=(E_0^g,p \mu) \times_{\Spec \F_{p^2}} S$ and $\rho_i^* \lambda_{i-1}=\lambda_i$ for all $i$;
    \item [(ii)] $\ker \rho_2$ and $\ker \rho_1$ 
    are $\alpha$-groups of respective ranks $g-1$ and $1$ over $S$;
    \item [(iii)] $\ker \lambda_1$ an $\alpha$-group of rank $2$ over $S$; that is, it is annihilated both by the Frobenius morphism $\ker \lambda_1 \to \ker \lambda_1^{(p)}$ and the Verschiebung morphism $\ker \lambda_1^{(p)} \to \ker \lambda_1$. 
\end{itemize}
Two objects $(Y_\bullet,\lambda_\bullet)$ and $(Y_\bullet',\lambda_\bullet')$ are said to be {\it isomorphic} if there exist isomorphisms $\alpha_i: (Y_i,\lambda_i)\isoto (Y_i',\lambda_i')$ for $i=0,1,2$ such that 
$\rho_{i}' \circ \alpha_i =\alpha_{i-1} \circ \rho_i$ for $i=0,1,2$ and $\alpha_2={\rm id}$.
\end{definition}

As proved in~\cite{lioort}, $\calP_\mu$ is representable by a projective scheme over $\F_{p^2}$, which we again denote by~$\calP_\mu$. Since $\deg (\rho_1 \circ \rho_2)$ has degree $p^g$ and $\lambda_2$ has degree $p^{2g}$, it follows that $(Y_0,\lambda_0)$ is a \emph{principally}
polarised supersingular abelian variety.\\ 

For any chain $(Y_{\bullet},\lambda_{\bullet}) \in \mathcal{P}_{\mu}(k)$, we consider the contravariant quasi-polarised Dieudonn{\'e} modules of the $(Y_i, \lambda_i)$. That is, let $(M_i, \langle\ , \rangle_{i})$ denote the Dieudonn{\'e} module of $(Y_i, \lambda_i)$ for $i= 0,1,2$. We write $N:=M_2\otimes_W W[1/p]$ for the common isocrystal and $\<\, ,\>$ for the common pairing on $N$. Then the Dieudonn{\'e} modules $M_i^t$ of the respective dual abelian varieties $Y_i^t$ can be identified with the dual lattices of $M_i$ in $N$ with respect to $\<\ ,\ \>$.  The quasi-polarisation $\<\, ,\>_\mu$ induced by the principal polarisation $\mu$ is a perfect pairing on $M_2$ and satisfies $\<\, ,\>_\mu=p\<\, ,\>$.   
In particular, we have $M_2^t = p M_2$ and 
\begin{equation}\label{eq:Mi}
p M_2 \subseteq_{g-1} M_1^t \subseteq_{1} M_0^t=M_0 \subseteq_{1} M_1 \subseteq_{g-1} M_2
\end{equation}
by definition, where the indices denote the respective lengths.
Moreover, denoting for any Dieudonn{\'e} module $M$ its quotient modulo $pM_2$ by $\ol M:=(M+pM_2)/pM_2$, we have that $\ol {M_i^t}$ is the orthogonal complement of $\ol M_i$ in the symplectic space $(\ol M_2, \<\, , \>_\mu)$.\\

We choose a basis $e_1, \ldots, e_g, f_1, \ldots, f_g$ for $M_2$ such that 
\begin{equation}\label{eq:FVeifi}
\mathsf{F} e_i = f_i, \qquad \mathsf{F}f_i = -pe_i, \qquad \mathsf{V} e_i = -f_i, \qquad \mathsf{V}f_i = pe_i
\end{equation}
for all $i = 1, \ldots, g$, and such that
\begin{equation}\label{eq:pairingeifi}
\langle e_i, f_j \rangle_\mu = \delta_{ij}=-\langle f_j, e_i \rangle_\mu, \quad \<e_i,e_j\>_\mu=\langle f_i,f_j\>_\mu=0
\end{equation}
for all $i,j = 1, \ldots, g$. We may assume that the basis $\{e_i,f_i\}$ also generates the \dieu module $M(E_0^g)$ of $E_0^g$ over ${\Z_{p^2}}=W(\F_{p^2})$.

\begin{definition}
    The $a$-number of a Dieudonn{\'e} module $M$ is $a(M) := \dim M/(\mathsf{F}, \mathsf{V})M$. The $a$-number of an abelian variety $X/k$ is $a(X) = \dim_k \mathrm{Hom}(\alpha_p, X)$. If $M$ is the Dieudonn{\'e} module of an abelian variety $X$, then $a(M) = a(X)$.
\end{definition}

We see from \eqref{eq:FVeifi} that 
\[
M_2/\mathsf{V}M_2 = \langle e_1, \ldots, e_g \rangle_k, \quad \sfV \ol M_2=\<f_1, \dots, f_g\>_k.
\]
Here to ease notation, we also write $e_i$ and $f_i$ for their images in the corresponding quotient if there is no risk of confusion. 
Since $\ker \rho_2$ is an $\alpha$-group by assumption, the quotient $M_2/M_1$ is annihilated by $\mathsf{F}$ and $\mathsf{V}$; in other words, we have
\begin{equation}\label{eq:VM2inM1}
\mathsf{F}M_2 = \mathsf{V}M_2 \subseteq M_1, 
\end{equation}
and since the $a$-number of $M_2$ equals $\dim M_2/(\mathsf{F}, \mathsf{V})M_2 = g$, we obtain from~\eqref{eq:Mi} that $\dim M_1/\mathsf{V}M_2=1$. We write 
\begin{equation}\label{eq:M1}
\begin{split}
M_1 &= \langle v \rangle_W + \mathsf{V}M_2,\\
v &:= [t_1]e_1 + \ldots + [t_g]e_g,
\end{split}
\end{equation}
where $t_i \in k$ and $[t_i] \in W$ denote their respective Teichm{\"u}ller lifts. For a commutative finite group scheme $G$ over $k$, let $M(G)$ denote its Dieudonn{\'e} module.

\begin{lemma}\label{lem:kerlambda1} The following three statements are equivalent:
\begin{enumerate}
    \item $\ker \lambda_1 \subseteq Y_1[\mathsf{F}]$;
    \item $\ker \lambda_1 \subseteq Y_1[\mathsf{V}]:=\ker (\sfV: Y_1\to Y_1^{(p^{-1})})$;
    \item the element $t=(t_1,\dots, t_g)$ satisfies $t_1^{p+1}+\dots + t_{g}^{p+1}=0$.
\end{enumerate}
\end{lemma}

\newpage
\begin{proof}\
\begin{enumerate}
    \item[(1) $\Leftrightarrow$ (2)]
    The inclusion $\ker \lambda_1 \subseteq Y_1[\mathsf{F}]$ of group schemes is equivalent to the surjection $M(Y_1[\mathsf{F}]) \twoheadrightarrow M(\ker \lambda_1)$ of Dieudonn{\'e} modules. Since $M(Y_1[\mathsf{F}]) = M_1/\mathsf{F}M_1$ and $M(\ker \lambda_1) = M_1/M_1^t$, we see that this is further equivalent to $\mathsf{F}M_1 \subseteq M_1^t$. Since $\mathsf{F} M_1 = \langle \mathsf{F} v \rangle_W + pM_2$ by~\eqref{eq:M1} and $pM_2 \subseteq M_1^t$ by~\eqref{eq:Mi}, we see that the inclusion holds if and only if 
    \begin{equation}\label{eq:Fermateq}
    \langle \mathsf{F} \bar{v}, \bar{v} \rangle_{\mu,k} = t_1^{p+1} + \dots + t_g^{p+1} = 0 
    \end{equation}
    for $\bar{v} = v \bmod pM_2$. Noting that also
    \[
    \langle \mathsf{V} \bar{v}, \bar{v} \rangle_{\mu,k} = t_1^{p^{-1}+1} + \dots + t_g^{p^{-1}+1}
    \]
    by replacing $\mathsf{F}$ by $\mathsf{V}$ in the above, it follows that statements (1) and (2) are equivalent.
    \item[(1) $\Leftrightarrow$ (3)] This follows from Equation~\eqref{eq:Fermateq}.    
\end{enumerate}
\end{proof}

\noindent In what follows, it will sometimes be convenient to use a different basis from $e_1, \ldots, e_g, f_1, \ldots, f_g$ for the module $\overline{M}_2=M_2/pM_2$. We define
\begin{equation}\label{eq:EiFi}
    \begin{split}
        E_1 = \sum_{i=1}^g t_i e_i, \quad E_2 = \sum_{i=1}^g t_i^p e_i, \quad \ldots, \quad E_{g-1} = \sum_{i=1}^g t_i^{p^{g-2}}e_i, \quad E_g = \sum_{i=1}^g t_i^{p^{-1}} e_i, \\
        F_1 = \sum_{i=1}^g t_i f_i, \quad F_2 = \sum_{i=1}^g t_i^p f_i, \quad \ldots, \quad F_{g-1} = \sum_{i=1}^g t_i^{p^{g-2}}f_i, \quad F_g = \sum_{i=1}^g t_i^{p^{-1}} f_i.
    \end{split}
\end{equation}

Thus, the set $\{E_1, \ldots, E_g, F_1, \ldots, F_g\}$ forms a $k$-basis of $\overline{M}_2$ if and only if the coefficient matrix
\begin{equation}\label{eq:T}
\mathbb{T} := \begin{bmatrix} 
t_1 & t_1^p & \ldots & t_1^{p^{g-2}} & t_1^{p^{-1}} \\
\vdots & \vdots & \ldots & \vdots & \vdots \\
t_g & t_g^p & \ldots & t_g^{p^{g-2}} & t_g^{p^{-1}} \\
\end{bmatrix}
\end{equation}
has nonzero determinant, which holds if and only if the projective point $t := (t_1:\ldots:t_g)$ is not contained in any $\mathbb{F}_p$-rational hyperplane. \\ 

Assume that $t$ satisfies $\sum_{i=1}^g t_i^{p+1}=0$ and that it is away from any $\mathbb{F}_p$-rational hyperplane.
From \eqref{eq:M1} it follows that, in this new basis,
\[
\overline{M}_1 = \langle E_1, F_1, \ldots, F_g \rangle_k,
\]
and using the pairings~\eqref{eq:FVeifi} it follows that we may write
\[
\overline{M}_1^t = \langle F_2, F_3 - t'_3 F_1,\dots ,F_{g-1} - t'_{g-1} F_1, F_g \rangle_k,
\]
for some $t'_3, \ldots, t'_{g-1} \in k$. In fact, 
\begin{equation}
    \label{eq:tiprime}
t_j'=\frac{\<E_1,F_j\>_\mu}{\<E_1,F_1\>_\mu}=\frac{\sum_{i=1}^g t_i^{p^{j-1}+1}}{\sum_{i=1}^g t_i^2}, \quad 3\le j\le g-1.
\end{equation}
Here we assume $\sum_{i=1}^g t_i^2\neq 0$ if $g\ge 4$; when $g=3$, we have $\overline{M}_1^t=\<F_2,F_3\>_k$.
Hence, by slight abuse of notation we obtain for any $g$ that 
\begin{equation}\label{eq:M1M1t}
\overline{M}_1/\overline{M}_1^t = \langle E_1, F_1 \rangle_k,
\end{equation}
and thus by~\eqref{eq:Mi} that
\begin{equation}\label{eq:M0bar}
   \overline{M}_0 = \langle E_1 + u F_1 \rangle_k + \overline{M}_1^t 
\end{equation}
for some $u \in k$, provided that $M_0\neq \sfV M_2$. 
\subsection{Geometry of $\mathcal{P}_{\mu}$}\

Analogous to \cite[\S 9.3]{lioort}, we consider truncations of the chains in $\mathcal{P}_{\mu}(k)$ via the forgetful map
\begin{equation}\label{eq:forget}
\left((Y_\bullet,\lambda_\bullet) =  (Y_2, \lambda_2) \xrightarrow{\rho_{2}}  (Y_1,\lambda_1) \xrightarrow{\rho_{1}} (Y_0,\lambda_0) \right) \mapsto \left( (Y_2, \lambda_2) \xrightarrow{\rho_{2}}  (Y_1,\lambda_1) \right).
\end{equation}
It follows from the observations in the previous subsection that this induces a map  
\[
\left( (M_0 \subseteq M_1 \subseteq M_2) \text{ such that } \mathsf{V}M_2 \subseteq M_1 \right) \mapsto (\mathsf{V}M_2 \subseteq M_1 \subseteq M_2)
\]
on chains of Dieudonn{\'e} modules, whose quasi-polarisations we have temporarily suppressed for ease of notation. Note that the condition $\mathsf{V}M_2 \subseteq M_1$ is \eqref{eq:VM2inM1}. The right hand side admits a projective (fine) moduli scheme, which is the Fermat hypersurface 
\[ \mathcal{F} := V(X_1^{p+1} + \ldots + X_g^{p+1}) \subseteq \mathbb{P}^{g-1}_{\F_{p^2}} \] 
by Lemma~\ref{lem:kerlambda1}, cf.~\cite[Lemma~3.7]{lioort}. Finally, the fibres of~\eqref{eq:forget} are $\alpha$-groups of rank 1 in the $\alpha$-group $\ker \lambda_1$ of rank $2$, by Definition~\ref{def:Pmu}.(ii) and (iii). Thus, we obtain a morphism 
\[ \pi: \mathcal{P}_{\mu} \to \mathcal{F} \] 
through which $\mathcal{P}_{\mu}$ is a $\mathbb{P}^1$-bundle over $\mathcal{F}$. As above, let $t = (t_1 : \ldots : t_g)$ denote a point on~$\mathcal{F}$ and denote the fibre over it by $\pi^{-1}(t) = \mathbb{P}^1_t$. Then a point $y$ on $\mathcal{P}_{\mu}(k)$ is determined by the pair of parameters $(t,u)$ with $t \in \mathcal{F}(k)$ and $u \in \mathbb{P}^1_t(k)$. By the description of $\calP_\mu$, we see that the geometry of $\calP_\mu$ does not depend on the choice of $\mu$. However, the arithmetic properties (for example, the automorphism group) of points in $\calP_\mu$ depend on both its coordinates $(t,u)$ and on $\mu$.

\begin{definition}
    The $a$-number of a point $y=(t,u)\in \calP_\mu$ is defined to be $a(y):=a(M_0)$, where 
    $(M_\bullet)$ is the chain of \dieu modules corresponding to $y$.
\end{definition}

Since $\ker \rho_2 \subseteq Y_2[\sfF]$ by definition and $Y_2/Y_2[\sfF]\simeq Y_2^{(p)}$,  there is a section $s: \mathcal{F} \to \mathcal{P}_{\mu}$ of~$\pi$, which sends 
\[ \left(  (Y_2, \lambda_2) \xrightarrow{\rho_{2}}  (Y_1,\lambda_1)\ \right)  \mapsto \  
\sfF=F_{Y_2/S}: (Y_2, \lambda_2) \xrightarrow{\rho_{2}}  (Y_1,\lambda_1) \xrightarrow{\rho_{1}} (Y_2,\lambda_2)^{(p)}. 
\]
In terms of the corresponding chains of \dieu modules in $\calP_\mu(k)$, the section $s: \calF(k) \to \calP_\mu(k)$ is given by  
\[ \left ( (M_1, \langle\ ,\ \rangle)\subseteq (M_2, \langle\ ,\ \rangle) \right )\ \mapsto \left ((\sfV M_2, \langle\ ,\ \rangle\subseteq (M_1, \langle\ ,\ \rangle)\subseteq (M_2,\langle\ ,\ \rangle) \right ), \quad M_0=\sfV M_2. \]
Set $T := s(\mathcal{F})$, viewed as the $\infty$-section of $\pi$, and $\calP_\mu':=\calP_\mu \setminus T$. Then $\pi: \calP_\mu'\to \calF$ equips $\calP_\mu'$ with the structure of an $\bbA^1$-bundle over $\calF$. Moreover, $a(y)=g$ for $y\in T$, since $\dim \mathsf{V}M_2 /(\mathsf{F},\mathsf{V})\mathsf{V}M_2 = \dim \mathsf{V} M_2 / pM_2 = g$.

\begin{remark}\label{rem:trivialization}
 The two-dimensional $k$-vector space $M_1/M_1^t$ for $t\in \calF(k)$ forms a rank two vector bundle $\calV$ over $\calF$ which is defined over $\F_{p^2}$. Over the open subset $U'\subseteq \calF$ away from all $\F_p$-rational hyperplanes and the conic $V(\sum_{i=1}^g X_i^{2})$, the elements $E_1$ and $F_1$ (defined in~\eqref{eq:EiFi}) define nowhere vanishing global sections of $\calV$ over $U'$, which generate the fibre $\calV_t$ at every $t\in U'$. Thus, $(E_1,F_1)$ gives an isomorphism of algebraic varieties over $\F_{p^2}$:
 \[ 
 (E_1, F_1): U'\times \A^2\isoto \calV\times_\calF U': (t, a_1,a_2) \mapsto (t, a_1E_1+a_2 E_2). 
 \]  
 This isomorphism induces a trivialisation $U'\times \bbA^1\isoto \calP'_\mu \times_\calF U'$ over $k$, which sends each $k$-point $(t,u)$ to the chain of \dieu modules $(M_i)_{i=0,1,2}$; here $M_i$ for $i=0,1$ are the respective lifts of $\ol M_i$ defined by Equation~\eqref{eq:M1} and Equation~\eqref{eq:M0bar} to $W$.
\end{remark} 

The following lemmas show how the $a$-numbers of the $M_i$ depend on the parameters $(t,u)$.

\begin{lemma}\label{lem:aM1}
    We have $a(M_1) \geq g-1$, and $a(M_1) = g$ if and only if $t \in \mathcal{F}(\mathbb{F}_{p^2})$.
\end{lemma}

\begin{proof}
    On the one hand, by~\eqref{eq:Mi} and~\eqref{eq:VM2inM1} we have $\mathsf{V}M_2 \subseteq_1 M_1$, which implies that $pM_2 \subseteq_1 \mathsf{F}M_1$ and $pM_2 \subseteq_1 \mathsf{V}M_1$. On the other hand, by~\eqref{eq:Mi}, we have $pM_2 \subseteq_{g+1} M_1$. Thus,
    \[
    pM_2 \subseteq_{\ell_1} (\mathsf{F}, \mathsf{V})M_1 \subseteq_{\ell_2} M_1
    \]
    where $\ell_1 + \ell_2 = g+1$ and $\ell_1 \in \{1,2\}$, so $\ell_2 \in \{g-1, g\}$, i.e. $a(M_1) \geq g-1$, as claimed. One also reads off that $a(M_1) = g$ if and only if $\mathsf{F}M_1 = \mathsf{V}M_1$,  if and only if the $k$-subspace $M_1/\sfV M_2$ is defined over $\F_{p^2}$, that is, $t \in \mathcal{F}(\mathbb{F}_{p^2})$.
\end{proof}

\begin{lemma}
    If $t \in \mathcal{F}(\mathbb{F}_{p^2})$ then $a(M_0) \geq g-1$, and $a(M_0) = g$ if and only if $u \in \mathbb{P}^1_t(\mathbb{F}_{p^2})$.
\end{lemma}

\begin{proof}
Since $t\in \calF(\F_{p^2})$, we have $a(M_1)=g$ by Lemma~\ref{lem:aM1}, i.e. the Dieudonn{\'e} module $M_1$ is superspecial. It follows that the subspace $M_0/M^t_1\subseteq M_1/M_1^t$ is defined over $\F_{p^2}$ if and only if $a(M_0)=g$, as well. 

We have $\sfF M_1= \sfV M_1$ by assumption, and $\dim \sfV M_1/\sfV M_0=\dim \sfF M_1/\sfF M_0=1$, so $\dim (\sfF,\sfV) M_1/(\sfF,\sfV)M_0=\dim \sfV M_1/(\sfF,\sfV)M_0 \in \{0,1\}$. 
Thus, $a(M_0) = \dim M_0/(\mathsf{F},\mathsf{V})M_0$ equals 
\[
\begin{split}
\dim M_0/(\mathsf{F},\mathsf{V})M_1 + \dim (\mathsf{F},\mathsf{V})M_1/(\mathsf{F},\mathsf{V})M_0 \\
= \dim M_1/(\mathsf{F},\mathsf{V})M_1 - 1 + \dim (\mathsf{F},\mathsf{V})M_1/(\mathsf{F},\mathsf{V})M_0 & \\
\geq g-1. 
\end{split}
\]
\end{proof}

\begin{lemma}
    When $t \not\in \mathcal{F}(\mathbb{F}_{p^2})$ and $y \not\in T$, we have $a(M_0) = g-2$.
\end{lemma}

\begin{proof}
We have $M_1^t \subseteq_1 M_0 \subseteq_1 M_1$ and $M_1^t \subseteq_1 \sfV M_2 \subseteq_1 M_1$. 
Then 
\[ \sfF M_0 \subseteq_1 \sfF M_1, \quad  \sfV M_0 \subseteq_1 \sfV M_1,  \quad\sfF M_1^t \subseteq_1 p  M_2 \subseteq_1 \sfF M_1, \quad \sfV M_1^t \subseteq_1 p M_2 \subseteq_1 \sfV M_1. \]
Since $M_1$ is not superspecial by our assumption that $t \not\in \mathcal{F}(\mathbb{F}_{p^2})$, we have that $\sfF M_1^t\neq \sfV M_1^t$ are distinct submodules of length 1 in $pM_2$, which implies that $\sfF M_1^t+\sfV M_1^t=pM_2$. 
On the other hand, since $M_0\neq \sfV M_2$ by our assumption that $y \not\in T$, we have $\sfF M_0 \neq pM_2$ and $\sfV M_0 \neq pM_2$; both $\sfF M_0$ and $\sfV M_0$ are submodules of lengh 1 of $\sfF M_1$ and $\sfV M_1$, respectively. Therefore, $\sfF M_0+pM_2=\sfF M_1$ and $\sfV M_0+pM_2=\sfV M_1$, and hence $\sfF M_0+\sfV M_0+pM_2=(\sfF, \sfV) M_1$.

Since $(\sfF,\sfV)M_0\supseteq (\sfF,\sfV)M_1^t=pM_2$, we have 
\begin{equation}
    \label{eq:FVM0}
    (\sfF,\sfV)M_0=(\sfF,\sfV)M_0+pM_2=(\sfF,\sfV)M_1.
\end{equation}
It follows that 
\[ pM_2 \subseteq_2 (\sfF,\sfV)M_1=(\sfF,\sfV)M_0 \subseteq M_0, \quad \text{ and } \quad pM_2 \subseteq_g M_0,  \]
and hence that $a(M_0)=g-2$.
\end{proof}

We summarise the results above in the following proposition.
\begin{proposition}\label{prop:aM0}
Let $y=(t,u) \in \mathcal{P}_{\mu}(k)$ and $(M_{\bullet}) = (M_0 \subseteq M_1 \subseteq M_2)$ the corresponding chain of Dieudonn{\'e} modules.
\begin{itemize}
    \item [(i)] $a(M_0)=g$ if and only if either $y\in T$ or $y\in \calP_\mu(\F_{p^2})$ (i.e. $t\in \calF(\F_{p^2})$ and $u \in \bbP_t(\F_{p^2})$);
    \item [(ii)] $a(M_0)=g-1$ if and only if $t\in 
    \calF(\F_{p^2})$ and $u\not\in \bbP^1(\F_{p^2})$;
    \item[(iii)] $a(M_0) = g-2$ if and only if $t\not\in \calF(\F_{p^2})$ and $y\notin T$.
\end{itemize}
\end{proposition}    

Figure~1 below provides a schematic summary of the results above.

\begin{figure}[h!]\label{fig:Pmu}
\begin{center}
\includegraphics[width=10cm]{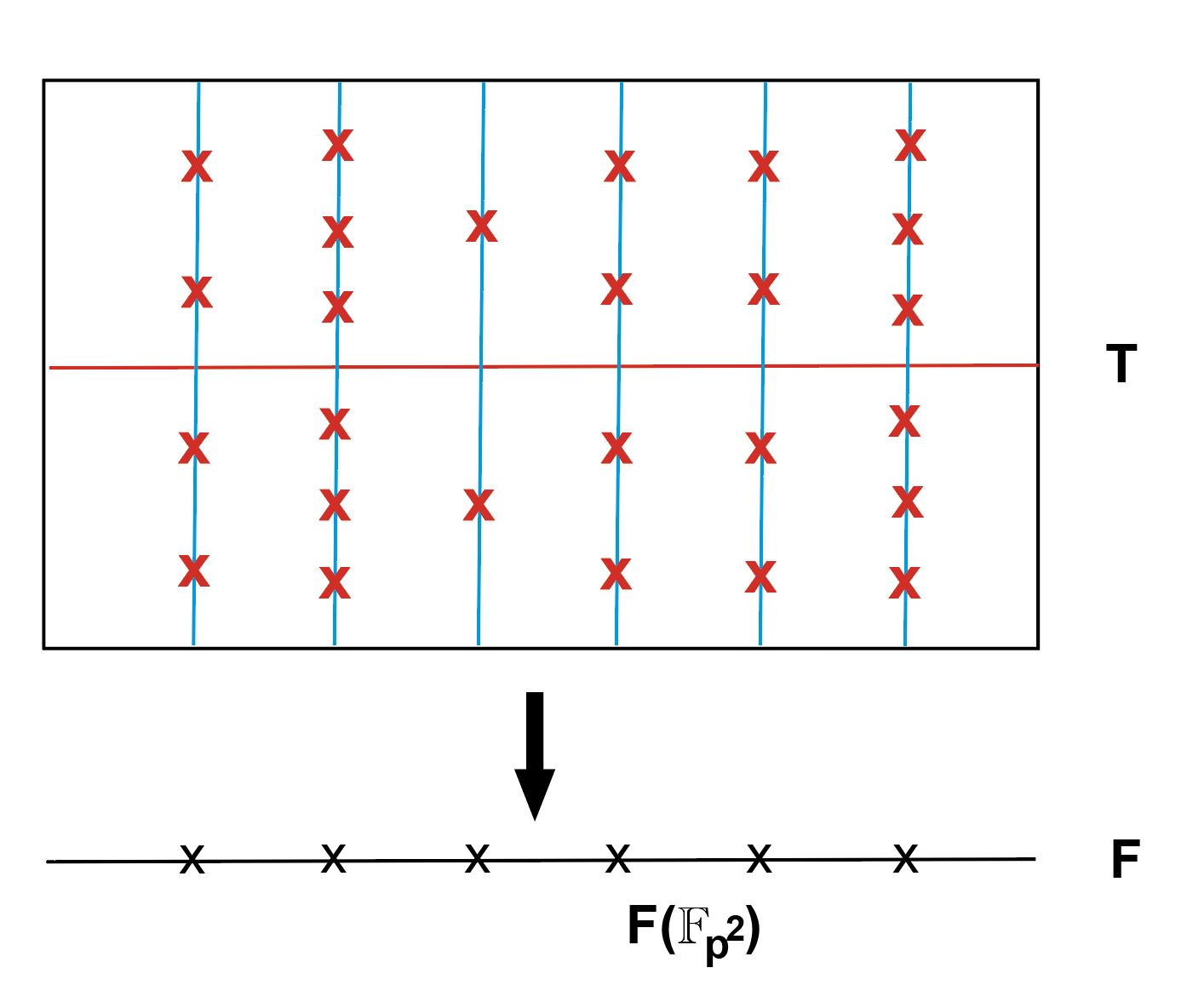}
\caption{A schematic picture of $\mathcal{P}_{\mu}$ as a $\mathbb{P}^1$-bundle over the Fermat hypersurface $\calF$. The points in red have $a$-number $g$, those in blue have $a$-number $g-1$; away from these points the $a$-number is $g-2$.}
\end{center}
\end{figure}

We end this section with the following observations. Let $M\subseteq N=M_2\otimes_W W[1/p]$ be a \dieu lattice. Denote by $\Phi$ the Zink operator which acts as $\Phi(M):=M+\sfV^{-1}\sfF M$. Then $M$ is superspecial if and only if $\Phi(M)=M$. Let $\wt M$ denote the smallest superspecial \dieu module containing $M$. We have $M\subseteq \Phi(M)$ and $\wt M=\Phi^n(M)$ for $n$ sufficiently large (in fact, if $n\ge g-1$). From the construction of $\wt M$, we have $\End_{\rm DM}(M)\subseteq \End_{\rm DM}(\wt M)$, and $\End_{\rm DM}(M)$ can be computed by 
\[ \End_{\rm DM}(M)=\{\alpha\in \End_{\rm DM}(\wt M): \alpha(M)\subseteq M\} . \]
Let $\calH\subseteq \bbP^{g-1}_{\F_{p^2}}$ denote the union of all hyperplanes in $\bbP^{g-1}_{\F_{p^2}}$ defined over $\F_{p^2}$. 
By slight abuse of notation, we let $\sigma$ also denote the Frobenius map $x \mapsto x^p$ on $k$.

\begin{lemma}\label{lm:minisog}
   Let $y=(t,u)\in \calP_\mu'$ be a point with $a(y)=g-2$, and $(M_\bullet)$ the chain of \dieu modules corresponding to y. Then $\wt M_0=M_2$ if and only if $t\notin\calH$.
\end{lemma}
\begin{proof}
    We write $M_2/\sfV M_2=\<e_1,\dots, e_g\>\otimes_{\F_{p^2}} k$, on which $\sigma^2$ acts via the second factor $k$. Any $k$-subspace of it is defined over $\F_{p^2}$ if and only if it is $\sigma^2$-stable, if and only if it is the image modulo $\mathsf{V} M_2$ of a superspecial \dieu module $M$ with $\sfV M_2\subseteq M\subseteq M_2$. 
    
    By~Proposition~\ref{prop:aM0}.(iii) and Equation~\eqref{eq:FVM0}, we have $\mathsf{V}\Phi(M_0)=(\sfF,\sfV)M_0=(\sfF,\sfV)M_1=V \Phi(M_1)$ and hence $\Phi(M_0)=\Phi (M_1)$. Thus, it suffices to prove that $\wt M_1=M_2$ if and only if $t\notin \calH$. Since $\wt M_1/\sfV M_2$ is the $k$-subspace spanned by $\bar v, \sigma^2(\bar v), \dots, \sigma^{2(g-1)}(\bar v)$, we see that $\wt M_1=M_2$ if and only if $\bar v, \sigma^2(\bar v), \dots, \sigma^{2(g-1)}(\bar v)$ form a $k$-basis of $M_2/\sfV M_2$, which happens if and only if $t\notin \calH$.
\end{proof}

\section{Endomorphisms and automorphisms}

We write $O_p:=\End(E_0)\otimes \Zp=\Z_{p^2}[\Pi]$, where $\Pi$ is a uniformiser of $O_p$ which satisfies the relations $\Pi^2=-p$ and $\Pi a =\sigma(a) \Pi$ for $a\in \Z_{p^2}$. Then $\End_{\rm DM}(M)=\Mat_g(\Z_{p^2}[\Pi])=\Mat_{g}(\Z_{p^2})\oplus \Mat_g(\Z_{p^2})\Pi$, where $\Pi$ acts the same as the Frobenius $\sfF$ on the basis $\{e_i, f_i\}$, namely, $\Pi(e_i)=f_i$ and $\Pi(f_i)=-pe_i$ for all $i=1, \ldots, g$. With respect to the basis $\{e_i, f_i\}$, we have an embedding of $\Zp$-algebras
\[ \End_{\rm DM}(M_2)=\Mat_g(\Z_{p^2}[\Pi]) \hookrightarrow \Mat_{2g}(\Z_{p^2}), \quad A+B\Pi \mapsto \begin{pmatrix}
    A & - p B^{\sigma} \\
    B & A^\sigma
\end{pmatrix}.\]
Reduction modulo $p$ gives a surjective homomorphism
\[ m_p: \End_{\rm DM}(M_2) \onto \Mat_g(\Z_{p^2}[\Pi])/(p)=\Mat_g(\F_{p^2}[\Pi])=\left \{\begin{pmatrix}
A & 0 \\
B & A^{(p)}\end{pmatrix} : A, B\in \Mat_g(\F_{p^2}) \right \}
\]
where $A^{(p)}=(a_{ij}^p)$ is the matrix obtained by raising the entries of $A=(a_{ij})$  to the $p$th power. Since $\langle\ ,\ \rangle_\mu=p\langle\ ,\ \rangle_2$, we have $\Aut_{\rm DM}(M_2,\langle\ ,\ \rangle_2)=\Aut_{\rm DM}(M_2,\langle\ ,\ \rangle_\mu)$.
The perfect alternating pairing $\<\ ,\ \>_\mu$ on $M_2$ induces the involution $\gamma=\left(\begin{smallmatrix} A & 0 \\ B & A^{(p)} \end{smallmatrix}\right)\mapsto \gamma^*:=\left(\begin{smallmatrix} (A^{(p)})^T  & 0 \\ -B^T & A^{T} \end{smallmatrix}\right)$ on $m_p(\mathrm{End}_{\mathrm{DM}}(M_2))$, where $A^T$ denotes the transpose of $A$. Then 
\[
m_p(\Aut_{\rm DM}(M_2,\<\ ,\ \>_2))=\{\gamma\in \GL_g(\F_{p^2}[\Pi]): \gamma^* \gamma=\bbI_g \}.
\]

In the following we assume that $t\in \calF(k)\setminus \calH$ and $y\notin T$. 

Then $M_1=M_0+\sfV M_2$ is uniquely determined by $M_0$. 
By Lemma~\ref{lm:minisog}, we have $\wt M_0=\wt M_1=M_2$ and hence $\End_{\rm DM}(M_0)\subseteq \End_{\rm DM}(M_1) \subseteq \End_{\rm DM}(M_2)$.
\begin{definition}
    \label{def:endt}
    For any $t=(t_1,\dots, t_g)\in k^g-\{0\}$ (viewed as a column vector), define an an $\F_{p^2}$-subalgebra of $\Mat_g(\F_{p^2})$ by 
\[ \End(t):=\{A\in \Mat_g(\F_{p^2}): A\cdot t\in k \cdot t  \}. \]
\end{definition}

\begin{lemma}
    If $t\in \bbP^{g-1}(k)\setminus \calH$, then the $\F_{p^2}$-algebra homomorphism $\End(t)\to k$, $A\mapsto \alpha_A$, the eigenvalue of $t$, is injective and $\End(t)\simeq \F_{p^{2m}}$ for some $m|g$.
\end{lemma}
\begin{proof}
If $\alpha_A=0$, then $A t=0$. Since $t\notin \calH$, the components $t_1,\dots,t_g$ are $\F_{p^2}$-linearly independent, so $A=0$ and the first statement follows. Particularly, $\End(t)\simeq \F_{p^{2m}}$ for some $m\ge 1$. Since $\End(t) \subseteq \Mat_g(\F_{p^2})=\End((\F_{p^2})^g)$, $(\F_{p^2})^g$ becomes a vector space over $\F_{p^{2m}}$ and hence $m|g$.  
\end{proof}

Let $\calQ\subseteq \bbP^{g-1}_{\F_{p^2}}$ be the union of all conics which are defined over $\F_{p^2}$. Then $\calH\subseteq \calQ$, since $\calQ$ contains degenerate conics $Q=H_1 H_2$ with $H_1, H_2$ in $\calH$. Set $\calF^0:=\calF\setminus \calQ$ and $\calP'_{\mu, \calF^0}:=\calP'_\mu \times_{\calF} \calF^0$. We see from Remark~\ref{rem:trivialization} that there is an isomorphism $\calP'_{\mu, \calF^0} \simeq \calF^0 \times \bbA^1$ of algebraic varieties over $k$.

\begin{proposition}\label{prop:mpEndM1}
    If $t\in \calF^0$, then 
    \[ m_p(\End_{\rm DM}(M_1))=\left\{\begin{pmatrix}
        a \bbI_g & 0 \\
        B & a^{p} \bbI_g
    \end{pmatrix}\in \Mat_{2g}(\F_{p^2}): a\in \F_{p^2}, B\in \Mat_g(\F_{p^2})\, \right\}, \]
    \[ m_p(\Aut_{\rm DM}(M_1,\<\ ,\ \>_1))=\left\{\begin{pmatrix}
        a \bbI_g & 0 \\
        B & a^{p} \bbI_g
    \end{pmatrix}\in \GL_{2g}(\F_{p^2}): a^{p+1}=1, B=B^T\, \right\}. \]
\end{proposition}
\begin{proof}
    Since $\sfV M_2\subseteq M_1\subseteq M_2$ and $\sfV M_2 $ is preserved by $\End_{\rm DM}(M_2)$, the subalgebra
     $m_p(\End_{\rm DM}(M_1))$ consists of elements $\gamma\in m_p(\End_{\rm DM}(M_2))$ which preserve $M_1/\sfV M_2=\<\bar v\> = \< t_1 e_1 + \ldots + t_g e_g \>$. So we have 
\[ m_p(\End_{\rm DM}(M_1))=\left\{\begin{pmatrix}
        A & 0 \\
        B & A^{(p)} 
    \end{pmatrix}: A\in \End(t) \, \right\}.\]    
Since $t\notin \calQ$ by assumption, the elements $t_i t_j$ for $1\le i\le j\le g$ are $\F_{p^2}$-linearly independent; equivalently, by specialisation at $t_g=1$, the elements $1, t_i , t_i t_j$ for $1\le i \le j\le g-1$ are $\F_{p^2}$-linearly independent. Applying \cite[Proposition~3.4]{karemaker-yu:SSEOOC} with $k_0=\F_{p^2}$, we get $\End(t)=\F_{p^2}$, which proves the first statement. 

The image of $\Aut_{\rm DM}(M_1,\<\ ,\ \>_1)$ under $m_p$ consists of elements $\gamma$ as above which further satisfy that
\[ \gamma^* \gamma =
\begin{pmatrix}
        a^p \bbI_g & 0 \\
        -B^T & a \bbI_g
    \end{pmatrix}
    \begin{pmatrix}
        a \bbI_g & 0 \\
        B & a^{p} \bbI_g
    \end{pmatrix} =
    \begin{pmatrix}
        a^{p+1} \bbI_g & 0 \\
        a(B-B^T) & a^{p+1} \bbI_g
    \end{pmatrix}=\bbI_{2g}. \]
This gives the conditions $a^{p+1}=1$ and $B=B^T$, as desired.
\end{proof}

Let $\gamma=\left ( \begin{smallmatrix}
    a \bbI_g & 0 \\
    B & a^p \bbI_g
\end{smallmatrix}\right )\in m_p(\Aut(M_1,\<\ ,\ \>_1))$. Then $\gamma$ preserves the two-dimensional vector space $\overline{M}_1/\overline{M}_1^t = \langle E_1, F_1 \rangle_k$. The matrix representing $\gamma$ with respect to the basis $\{E_i,F_i\}$ is
\[ \gamma=\begin{pmatrix}
    a \bbI_g & 0 \\
    \bbT^{-1} B \bbT & a^p \bbI_g 
\end{pmatrix}.
\]
We write $\bbT^{-1} B \bbT=(d_{ij})$. Then
\[ \gamma(F_1)=a^p F_1, \quad \gamma(E_1)=a E_1 +\sum_{j=1}^g d_{j1} F_j\equiv aE_1+(d_{11}+\sum_{j=3}^{g-1} t_j' d_{j1}) F_1 \mod \ol M_1^t, \]
where $t_j'$ are the functions of $t_i$ given in Equation~\eqref{eq:tiprime}.
Thus, the  matrix representing $\gamma$ on $\ol M_1/\ol M_1^t$ with respect to the basis $\{E_1,F_1\}$ is given by
\[ \gamma=\begin{pmatrix}
    a & 0 \\
    d_B & a^p
\end{pmatrix}, \quad \text{ where } \quad d_B:=d_{11}+\sum_{j=3}^{g-1} t_j' d_{j1}. \]
We write $\bbT^*:=\det(\bbT) \bbT^{-1}=(t_{ij}^*)$ and $\bbT=(t_{ij})$, then $d_{ij}$ is the $(i,j)$-component of the matrix $\det(\bbT)^{-1} \bbT^* B \bbT$, which is $\det(\bbT)^{-1}\sum_{l,l'} t_{i l}^* b_{l, l'} t_{l' j}$. In particular, each $d_{ij}$ is a linear combination of $b_{l,l'}$ for $1\le l,l'\le g$ with coefficients which are rational functions in $t_1^{{1/p}},\dots, t_g^{{1/p}}$, and so is~$d_B$.
Since $B=B^T$ is symmetric by assumption, we may write $d_B=\sum_{1\le l \le l' \le g} b_{l, l'} e_{l, l}'$, where $e_{l, l'}$ is a rational function in $t_1^{{1/p}},\dots, t_g^{{1/p}}$, and we define an $\F_{p^2}$-vector subspace of $k$ by
\begin{equation}
    \label{eq:Dt}
   \calD_t:=\<e_{l,l'}(t) : 1\le l \le l'\le g\>_{\F_{p^2}}\subseteq k.  
\end{equation}

In order to have $\gamma \in m_p(\Aut(M_0,\<\ ,\ \>_0))$, the element $\gamma$ must satisfy $\gamma(\ol M_0)\subseteq \ol M_0$, that is,
\begin{equation}\label{eq:gammaM0bar}
\begin{pmatrix}
    a & 0 \\
    d_B & a^p
\end{pmatrix} \begin{pmatrix}
    1 \\
    u
\end{pmatrix}=\begin{pmatrix}
    a \\
    d_B+a^p u
\end{pmatrix}=a \begin{pmatrix}
    1 \\
    u
\end{pmatrix}, \quad \text{ i.e. } \quad d_B+(a^p-a)u=0.
\end{equation}

\begin{definition}\label{def:D}
   Let $\calD \subseteq \calP'_{\mu, \calF^0}$ be the horizontal divisor defined by 
    \[  \calD:=\{(t,u)\in  \calP'_{\mu, \calF^0}: u\in \calD_t \}. \]
\end{definition}

\begin{theorem}\label{thm:OCPmu}
    Let $y=(t,u)\in \calP'_\mu(k)$ and $(Y_\bullet, \lambda_\bullet)$ be the chain of polarised supersingular abelian varieties corresponding to $y$. If $y\in \calP'_{\mu, \calF^0}\setminus \calD$ and $p>2$, then $\Aut(Y_0,\lambda_0)=\{\pm 1\}$. 
\end{theorem}
\begin{proof}
    Let $(M_i,\<\ ,\ \>_i)$ be the chain of quasi-polarised \dieu modules associated to $(Y_\bullet, \lambda_\bullet)$. Since $p>2$, it follows from Lemma~\cite[Lemma~6.16]{karemaker-yu:SSEOOC} that the group homomorphism 
    \[ m_p: \Aut(Y_0,\lambda_0) \to m_p(\Aut(M_0,\<\ ,\ \>_0)) \]
    is injective. Thus, it suffices to show that the latter group is $\{\pm 1\}$.

    Let $\gamma \in m_p(\Aut(M_0,\<\,,\>_0))$. Since $u\notin \calD_t$, Equation~\eqref{eq:gammaM0bar} shows that then $a^p-a=0$ and $d_B=0$, since $d_B\in \calD_t$. So we have $a^2=1$ and $a=\pm 1$, and hence $\gamma=\pm \begin{pmatrix}
         \bbI_g & 0 \\
        B & \bbI_g
    \end{pmatrix}$ with $B=B^{T}$. We must show that $B=0$.

    Observe that $\gamma(\ol M_0)\subseteq \ol M_0$ if and only if 
    $\gamma_0(\ol M_0)\subseteq \ol M_0$, where $\gamma_0=\begin{pmatrix}
        0 & 0 \\
        B & 0
    \end{pmatrix}$. We have 
    $\gamma_0(f_j)=0$ and $\gamma_0(e_j)=\sum_{i=1}^g b_{ij} f_i$ for all $j = 1, \ldots, g$.
    Since $\ol M_1^t \subseteq \sfV \ol M_2=\<f_1,\dots, f_g\>_k$ and $\sfV \ol M_2$ is a maximal isotropic subspace of $\ol M_2$, we obtain
    \begin{equation}\label{eq:M1tbar}
        \ol M_1^t=\{x\in \sfV \ol M_2: \<E_1, x\>_\mu=0\,\}=\left \{\sum_{i=1}^g a_i f_i\, :\, \sum_{i=1}^g a_i t_i=0\, \right\}. 
    \end{equation}
    Since $\gamma_0 (\ol M_0) \subseteq \gamma_0 (\ol M_2) \subseteq \sfV \ol M_2$, we have $\gamma_0(\ol M_0)\subseteq \ol M_0 \cap \sfV \ol M_2$. We have $\dim \ol M_0/\ol {M_1^t}=\dim \sfV \ol M_2/ \ol {M_1^t}=1$. Since $\ol M_0\neq \sfV \ol M_2$, we have $\ol M_0 \cap \sfV \ol M_2=\ol M_1^t$. Thus, $\gamma_0(\ol M_0)=\<\gamma_0 (E_1)\>\subseteq \ol M^t_1$, that is, 
    \[ \gamma_0(E_1)=\gamma_0 \left( \sum_{j=1}^g t_j e_j \right)=\sum_{j=1}^g \sum_{i=1}^g t_j b_{ij} f_i\in \ol M_1^t. \]
    By Equation~\eqref{eq:M1tbar}, we therefore have 
    \[ \sum_i \sum_j t_{j} b_{ij} t_i=\sum_{i,j} b_{ij} t_i t_j=0.\]
    Since the elements $t_i t_j$ are $\F_{p^2}$-linearly independent, we get $b_{ii}=0$ and $b_{ij}+b_{ji}=0$ for $1\le i\le j\le g$.
    Since $b_{ij}=b_{ji}$ and $p>2$, we also get $b_{ij}=0$ for all $i<j$. That is, $B=0$, as desired.   
\end{proof}

Recall that $\mathcal{S}_g$ denotes the supersingular locus of $\mathcal{A}_g$.
Let $\pr_0: \calP_\mu \to \calS_g$ be the natural projection $(Y_\bullet, \lambda_\bullet)\mapsto (Y_0,\lambda_0)$ and define 
\begin{equation}
    \label{eq:Tg&U}
     \calT_g:=\bigcup_{\mu \in {\rm Pol}(E_0^g)} \pr_0(\calP_\mu)\quad{and} \quad \calU:=\bigcup_{\mu \in {\rm Pol}(E_0^g)} \pr_0(\calP'_{\mu,\calF^0} \setminus \calD).
\end{equation}
As $\calP_\mu$ is projective, the morphism $\pr_0$ is proper and hence $\calT_g\subseteq \calS_g$ is a closed subvariety. Moreover, $\calU$ is an open dense subset; we prove this in the next theorem.
  
\begin{theorem}\label{thm:OCTg}
   The closed subvariety $\calT_g\subseteq \calS_g$ is prime-to-$p$ Hecke invariant and has equi-dimension $g-1$. The open dense subset $\calU$ is prime-to-$p$ Hecke invariant with $a$-number $g-2$ everywhere. Moreover, every principally polarised supersingular abelian variety  $(X,\lambda)$ in $\calU(k)$  has automorphism group $\{\pm 1\}$ if $p>2$.  
\end{theorem}
\begin{proof}
It is clear that $\calT_g$ is $\ell$-adic Hecke invariant for any prime $\ell\neq p$, as so is the set ${\rm Pol}(E_0^g)$ of principal polarisations. Moreover, $\calU$ is also $\ell$-adic Hecke invariant, as the construction of the complement $(\calP'_{\mu,\calF^0} \setminus \calD)^c$ of $\calP'_{\mu,\calF^0} \setminus \calD$ in $\calP_\mu'$ only uses properties of chains of \dieu modules. By Chevalley's Theorem, $\calU$ is a constructible subset a priori. We now show that $\calU$ is a open subset. Consider the proper morphism $\pr_0$.
As $\pr_0$ contracts the section $T$ to a point $p_\mu$ and $T$ is the preimage of $p_\mu$, the restricted morphism 
\[ \pr_0: \calP'_\mu \to \pr_0(\calP'_\mu)=\pr_0(\calP_\mu)\setminus \{p_\mu\} \] 
is again proper. Since $(\calP'_{\mu,\calF^0} \setminus \calD)^c$ is closed in $\calP_\mu'$, so is its image in $\pr_0(\calP_\mu')$, by the properness of $\pr_0$. As a result, $\calU$ is the complement of a closed subset in the open dense subset $\cup_{\mu} \pr_0 (\calP'_\mu)$ and hence an open dense subset in $\calT_g$. The remaining properties follow from Proposition~\ref{prop:aM0}(iii) and Theorem~\ref{thm:OCPmu}.  
\end{proof}

\noindent {\bf Proof of Theorem~\ref{thm:OCpodd}.}
The case where $g=2$ has been proved by Ibukiyama~\cite{ibukiyama} and Karemaker--Pries~\cite{karemaker-pries}, so we may assume that $g\ge 3$.
Let ${\rm Irr}(\calS_g)$ denote the set of (geometric) irreducible components of $\calS_g$, and let $V_1$ be an irreducible component of $\calS_g$ containing an irreducible component $\calW:=\pr(\calP'_{\mu, \calF^0} \setminus \calD)$ of $\calU$.
Since the $\ell$-adic Hecke action operates transitively on ${\rm Irr}(\calS_g)$ for any prime $\ell\neq p$ \cite[Proposition 6.22]{karemaker-yu:SSEOOC}, for any $V_1'\in {\rm Irr}(\calS_g)$, there exist an $\ell$-power Hecke correspondence $\calH_{\ell}$ over $\calS_g$
\[ \begin{tikzcd}
    & \calH_\ell \arrow[ld, "\pr_1"], \arrow[rd, "\pr_2"] & \\
    \calS_g &            & \calS_g
\end{tikzcd}
\]
and an irreducible component $\wt V_1$ of $\calH_{\ell}$ such that $\pr_1(\wt V_1)=V_1$ and $\pr_2(\wt V_1)=V_1'$. Since $\calU$ is $\ell$-adic Hecke invariant by Theorem~\ref{thm:OCTg}, the intersection $\calU\cap V_1\supseteq \pr_2(\pr_1^{-1}(\calW))\cap V_1'$ is non-empty. For any polarised abelian variety $(X',\lambda')$ corresponding to a $k$-point $s$ in $\calU\cap V_1'$, one has $\Aut(X',\lambda')=\{\pm 1\}$. Since there is a finite base change $\eta''$ of the generic point $\eta'$ of $V_1'$ so that $\Aut(X_{\bar \eta'},\lambda_{\bar \eta'})=\Aut(X_{\bar \eta''},\lambda_{\bar \eta''})$ and $s$ is a specialisation of $\eta''$, one has $\Aut(X_{\bar \eta'},\lambda_{\bar \eta'})\subseteq \Aut(X',\lambda')=\{\pm 1\}$. \qed

\begin{remark}
    In~\cite{karemaker-yobuko-yu}, Yobuko and the present authors study the mass stratification on the moduli space $\calP_{3,\mu}$ (which coincides with $\calP_\mu$ here for $g=3$) of $3$-dimensional polarised flag type quotients, and show that in the maximal mass stratum, every polarised abelian variety has automorphism group $\{\pm 1\}$ when $p\neq 2$. In the present paper, keeping the same assumption on $p$, we show that every member in $\calP_{\mu, \calF^0}'\setminus \calD$ also has automorphism group $\{\pm 1\}$. We now claim that when $g=3$ (and $p\neq 2)$, both open subsets coincide, though both their constructions look different. \\
    \indent In \emph{loc.cit.}~(and using the notation there), the authors define, for each $t\in C^0:=C\setminus C(\F_{p^2})$, where $C$ is the Fermat curve of degree $p+1$ in $\mathbb{P}^2$, a map $\psi_t:S(\F_p^2)\to k$, which sends each $3\times 3$ symmetric matrix $S$ to the $(1,1)$-entry of $\bbT^{-1}S \bbT$, where $\mathbb{T}$ is as in~\eqref{eq:T} for $g=3$. They then define a function $d: C^0(k) \to \bbN$ by setting $d(t):=\dim_{\F_{p^2}}{\rm Im}\, \psi_t$, which take values in $\{3,4,5,6\}$ (cf.\cite[Proposition~5.13]{karemaker-yobuko-yu}). In addition, a horizontal divisor $\calD\subseteq \calP_{C^0} = \mathcal{P}_{\mu} \times_C C^0$ is constructed so that its fibre $\calD_t$ at $t$ satisfies $\calD_t={\rm Im}\, \psi_t$ (cf.~\cite[Definition~5.16]{karemaker-yobuko-yu}), which agrees with the horizontal divisor $\calD$ defined in the present paper.  \\  
    \indent Using the mass formula (cf.~\cite[Theorem B]{karemaker-yobuko-yu}), the maximal mass stratum consists of all points $(t,u)\in \calP_{C^0}$ with $d(t)=6$ and $u\notin \calD_t$. On the other hand, one computes that $d(t)=\dim_{\F_{p^2}}\<t_i t_j; 1\le i\le j\le 3\>_{\F_{p^2}}$ (below \cite[Lemma~5.14]{karemaker-yobuko-yu}). Thus, $d(t)=6$ if and only if the elements $t_i t_j$, for $1\le i\le j\le 3$, are $\F_{p^2}$-linearly independent, or equivalently $t\notin \calQ$. This observation shows that both open subsets coincide.  
\end{remark}

\bibliographystyle{amsplain}
\bibliography{OortConjecture}

@article {COirr,
    AUTHOR = {Chai, Ching-Li and Oort, Frans},
     TITLE = {Monodromy and irreducibility of leaves},
   JOURNAL = {Ann. of Math. (2)},
  FJOURNAL = {Annals of Mathematics. Second Series},
    VOLUME = {173},
      YEAR = {2011},
    NUMBER = {3},
     PAGES = {pp.~1359--1396},
      ISSN = {0003-486X},
       DOI = {10.4007/annals.2011.173.3.3},
       URL = {https://doi.org/10.4007/annals.2011.173.3.3},
}

@article {edixhoven-moonen-oort,
    AUTHOR = {Edixhoven, Sebastiaan and Moonen, Ben and Oort, Frans},
     TITLE = {Open problems in algebraic geometry},
   JOURNAL = {Bull. Sci. Math.},
  FJOURNAL = {Bulletin des Sciences Math\'{e}matiques},
    VOLUME = 125,
      YEAR = 2001,
    NUMBER = 1,
     PAGES = {pp.~1--22},
      ISSN = {0007-4497},
       DOI = {10.1016/S0007-4497(00)01075-7},
       URL = {https://doi.org/10.1016/S0007-4497(00)01075-7},
}

@article {ibukiyama,
    AUTHOR = {Ibukiyama, Tomoyoshi},
     TITLE = {Principal polarizations of supersingular abelian surfaces},
   JOURNAL = {J. Math. Soc. Japan},
  FJOURNAL = {Journal of the Mathematical Society of Japan},
    VOLUME = {72},
      YEAR = {2020},
    NUMBER = {4},
     PAGES = {pp.~1161--1180},
      ISSN = {0025-5645},
       DOI = {10.2969/jmsj/82528252},
       URL = {https://doi-org.proxy.library.uu.nl/10.2969/jmsj/82528252},
}

@article {karemaker-yobuko-yu,
    AUTHOR = {Karemaker, Valentijn and Yobuko, Fuetaro and Yu, Chia-Fu},
     TITLE = {Mass formula and {O}ort's conjecture for supersingular abelian threefolds},
   JOURNAL = {Adv. Math.},
  FJOURNAL = {Advances in Mathematics},
    VOLUME = {386},
      YEAR = {2021},
     PAGES = {Paper No. 107812, 52},
      ISSN = {0001-8708},
       DOI = {10.1016/j.aim.2021.107812},
       URL = {https://doi-org.proxy.library.uu.nl/10.1016/j.aim.2021.107812},
}

@book {lioort,
    AUTHOR = {Li, Ke-Zheng and Oort, Frans},
     TITLE = {Moduli of supersingular abelian varieties},
    SERIES = {Lecture Notes in Mathematics},
    VOLUME = {1680},
 PUBLISHER = {Springer-Verlag, Berlin},
      YEAR = {1998},
     PAGES = {iv+116},
      ISBN = {3-540-63923-3},
}

@article {karemaker-pries,
    AUTHOR = {Karemaker, Valentijn and Pries, Rachel},
     TITLE = {Fully maximal and fully minimal abelian varieties},
   JOURNAL = {J. Pure Appl. Algebra},
  FJOURNAL = {Journal of Pure and Applied Algebra},
    VOLUME = {223},
      YEAR = {2019},
    NUMBER = {7},
     PAGES = {pp.~3031--3056},
      ISSN = {0022-4049,1873-1376},
       DOI = {10.1016/j.jpaa.2018.10.007},
       URL = {https://doi.org/10.1016/j.jpaa.2018.10.007},
}

@misc{karemaker-yu:SSEOOC,
      title={{S}upersingular {E}kedahl-{O}ort strata and {O}ort's conjecture}, 
      author={Valentijn Karemaker and Chia-Fu Yu},
      year={2026},
      eprint={2406.19748},
      archivePrefix={arXiv},
      primaryClass={math.NT},
      url={https://arxiv.org/abs/2406.19748}, 
      note={arXiv:2406.19748},
}

@misc{viehmann:oort,
      title={Oort's conjecture on automorphisms of generic supersingular abelian varieties}, 
      author={Eva Viehmann},
      year={2026},
      eprint={2603.06033},
      archivePrefix={arXiv},
      primaryClass={math.AG},
      url={https://arxiv.org/abs/2603.06033}, 
      note={arXiv:2603.06033},
}

\end{document}